\documentclass[12pt]{amsart}
\usepackage{amsmath,amsthm,amsfonts,amssymb,amscd,euscript}

\newlength{\defbaselineskip}
\theoremstyle{plain}

\theoremstyle{definition}

\newcommand{\gchoose}[2]{\left[\begin{array}{c}{#1 } \\{#2} \end{array}\right]  }

\theoremstyle{plain}
\newtheorem{thm}{Theorem}

\newtheorem{lem}[thm]{Lemma}
\newtheorem{cor}[thm]{Corollary}

\newtheorem{conj}[thm]{Conjecture}

\theoremstyle{definition}
\newtheorem{defn}[thm]{Definition}

\newtheorem{exmp}[thm]{Example}

\numberwithin{equation}{section}
\begin{document}





\title[Sign-coherency of symmetric polynomials]{Strong sign-coherency of certain symmetric polynomials, with application to cluster algebras}
\author{Kyungyong Lee}
\thanks{Research partially supported by NSF grant DMS 0901367.}


\address{Department of Mathematics, University of Connecticut, Storrs, CT 06269}
\email{{\tt kyungl@purdue.edu}}
\begin{abstract}
For each positive integer $n$, we define a polynomial in the variables $z_1,\cdots,z_n$ with coefficients in the ring $\mathbb{Q}[q,t,r]$ of polynomial  functions of three parameters $q, t, r$. These polynomials naturally arise in the context of cluster algebras.  We conjecture that they are symmetric polynomials in $z_1,\cdots,z_n$, and that their expansions in terms of monomial, Schur, complete homogeneous, elementary and power sum symmetric polynomials are sign-coherent. 
\end{abstract}

 \maketitle
 



\section{introduction}
The purpose of this note is to introduce an interesting family of (conjecturally symmetric) polynomials, which naturally arises in the context of cluster algebras.  For each positive integer $n$, we define a polynomial in the variables $Z=\{z_1,\cdots,z_n\}$ with coefficients in the ring $\mathbb{Z}[q,t,r]$ of polynomial  functions of the three parameters $q, t, r$. In Section 2, we will explain their connection to cluster algebras.  In order to define them, we need some notations.

Let $S=\{1,2,3,\cdots, n \}$, and $\mathcal{S}$ be the collection of all set partitions of $S$. For any subset  $S_j$, say $\{v_1,\cdots,v_m\}$, of $S$, let 
$$\sigma(S_j):=z_{v_1}+\cdots+z_{v_m}.$$ 
When $P=S_1\sqcup S_2\sqcup \cdots \sqcup S_k$ is a partition of $S$, we denote $k$ by $|P|$.  We give an order on $\{S_1,\cdots, S_k\}$ as follows : $S_i < S_j$ if and only if $\min S_i < \min S_j$. For any partition of $S$, we assume that $S_1< \cdots <S_k$. Let $$d(j,i):=|S_j|\sigma(S_i)-|S_i|\sigma(S_j).$$   

\begin{defn}
Let $T: \{1, 2,\cdots, m\} \longrightarrow \mathbb{N}$ be any strictly increasing function. Let $\mathcal{W}$ be the set of all permutations of $\{2,\cdots,m\}$. We define $e(T)$ by $$e(T)=\left\{\begin{array}{ll} 1, & \text{ if }m=1\\  \frac{1}{m!}\sum_{(p_2,\cdots, p_m)\in \mathcal{W}} \prod_{i=1}^{m-1}  (r(z_{T(1)}+z_{T(p_2)}+\cdots+z_{T(p_{i})}-iz_{T(p_{i+1})})-i),  & \text{ otherwise}.  \end {array}\right.$$
Let $S_j$ be any subset, say $\{v_1<\cdots<v_m\}$, of $S=\{1,2,3,\cdots, n\}$. It naturally gives rise to the  increasing function $T$, that is $T(i)=v_i$ $(1\leq i\leq m)$. By abuse of notation, set $e(S_j)=e(T).$ For any partition $P (=S_1\sqcup S_2\sqcup \cdots \sqcup S_k)$ of $S$, we define $$e(P):=\prod_{j=1}^k e(S_j).$$ \qed
\end{defn}
\begin{exmp}
$$\aligned &e(\{1,2,4\}\sqcup\{3,5\})=e(\{1,2,4\})e(\{3,5\})\\
&=\frac{1}{3!}\big((r(z_1-z_2)-1)(r(z_1+z_2-2z_4)-2)+(r(z_1-z_4)-1)(r(z_1+z_4-2z_2)-2)  \big)\\
&\,\,\,\,\,\,\, \times\frac{1}{2!}(r(z_3-z_5)-1).\endaligned$$
\end{exmp}
Now we are ready to define our polynomial. 

\begin{defn}
$$EC_n (Z; q, t, r):=n!\sum_{P(=S_1\sqcup S_2\sqcup \cdots \sqcup S_{|P|})\in \mathcal{S}} e(P) \prod_{j=1}^{|P|} \left(-|S_j|q + t\sigma(S_j) -r\sum_{i=1}^{j-1} d(j,i)\right).$$
\end{defn}

We call this the $`EC$'-polynomial, because it naturally occurs in the expression of the `Euler Characteristic' of a cell of a certain quiver Grassmannian. Despite its non-symmetric appearance, the $EC$-polynomial is expected to be symmetric in $z_1,\cdots,z_n$. Furthermore, it seems to have surprising positivity properties.  

\begin{conj}\label{symconj}
For every positive integer $n$, $EC_n (Z; q, t, r)$ is a symmetric polynomial of $z_1,\cdots,z_n$.
\end{conj}

\begin{conj}\label{signcohconj}
Suppose that Conjecture~\ref{symconj} is true. When expanded in terms of monomial $($resp. Schur, complete homogeneous, elementary and power sum$)$ symmetric polynomials, $EC_n (Z; q, t, r)$ is sign-coherent, i.e., for any partition $\lambda$, its  coefficient of $m_\lambda$ $($resp. $s_\lambda,h_\lambda, e_\lambda, p_\lambda)$  is a polynomial of $q,t,r$ with either all nonnegative or all nonpositive integer coefficients, and the sign depends only on the parity of the number of parts in $\lambda$. 
\end{conj}

Conjecure~\ref{symconj}  and Conjecure~\ref{signcohconj} are true for $n\leq 7.$

\begin{exmp}\label{ECexample}
$$EC_1 (Z; q, t, r)=e(\{1\}) (-q+tz_1)=-q+tz_1.$$
$$\aligned &EC_2 (Z; q, t, r)\\ 
&=2e(\{1\}\sqcup\{2\}) (-q+tz_1)(-q+tz_2-r(z_1-z_2)) + 2e(\{1,2\})(-2q+t(z_1+z_2)) \\
&=2e(\{1\})e(\{2\}) (-q+tz_1)(-q+tz_2-r(z_1-z_2)) + 2e(\{1,2\})(-2q+t(z_1+z_2)) \\
&=2 (-q+tz_1)(-q+tz_2-r(z_1-z_2)) + 2\frac{r(z_1-z_2)-1}{2}(-2q+t(z_1+z_2)) \\
&=-tr(z_1^2+z_2^2) + (2t^2+2tr)z_1z_2 + (-2qt-t)(z_1+z_2)+2q^2+2q.
\endaligned$$
$$\aligned &EC_3 (Z; q, t, r)\\ &=6e(\{1\}\sqcup\{2\}\sqcup\{3\})(-q+tz_1)(-q+tz_2-r(z_1-z_2))(-q+tz_3-r(z_1-z_3+z_2-z_3))\\
&\,\,\,\,\,\, +6e(\{1,2\}\sqcup\{3\})(-2q+t(z_1+z_2))(-q+tz_3-r(z_1+z_2-2z_3))\\
 &\,\,\,\,\,\, +6e(\{1,3\}\sqcup\{2\})(-2q+t(z_1+z_3))(-q+tz_2-r(z_1+z_3-2z_2))\\
  &\,\,\,\,\,\, +6e(\{1\}\sqcup\{2,3\})(-q+tz_1)(-2q+t(z_2+z_3)-r(2z_1-z_2-z_3))\\
  &\,\,\,\,\,\, +6e(\{1,2,3\})(-3q+t(z_1+z_2+z_3))\\
  &= 2tr^2(z_1^3+z_2^3+z_3^3) + (-3t^2r-3tr^2)(z_1^2z_2+\cdots+z_3^2z_1)+ (6t^3+18t^2r+12tr^2)z_1z_2z_3\\
  &\,\,\,\,\,\, +(6qtr + 6tr)(z_1^2+z_2^2+z_3^2) + (-6qt^2-6qtr-6t^2-6tr)(z_1z_2+z_2z_3+z_3z_1)\\
   &\,\,\,\,\,\, + (6q^2 t + 12qt + 4t)(z_1+z_2+z_3)-6q^3-18q^2-12q\\
   &= 2tr^2 s_{(3)}Z + (-3t^2r-5tr^2)s_{(2,1)}Z + (6t^3+24t^2r+20tr^2)s_{(1,1,1)}Z\\
   &\,\,\,\,\,\, +(6qtr + 6tr)s_{(2)}Z+ (-6qt^2-12qtr-6t^2-12tr)s_{(1,1)}Z\\
   &\,\,\,\,\,\, + (6q^2 t + 12qt + 4t)s_{(1)}Z-6q^3-18q^2-12, 
\endaligned$$where $s_{\lambda}Z$ are Schur symmetric polynomials.  \qed
\end{exmp}

\section{The Euler Characteristic of quiver Grassmannians}\label{cluster}
In this section, we explain how $EC$-polynomials arise in the context of cluster algebras. 

First, we define cluster algebras. To avoid too much distraction, we restrict ourselves to the rank 2 case. Let $b,c$ be positive integers and $x_1, x_2$ be indeterminates. The (coefficient-free) \emph{cluster algebra} $\mathcal{A}(b,c)$ is the subring of the field $\mathbb{Q}(x_1,x_2)$ generated by the elements $x_m$, $m\in \mathbb{Z}$ satisfying the recurrence relations:
$$
x_{n+1} =\left\{ \begin{array}{cl}
(x_n^b +1)/{x_{n-1}} & \text{ if } n \text{ is odd,} \\
\text{ } & \text{ }\\
{(x_n^c +1)}/{x_{n-1}} & \text{ if } n \text{ is even.}
\end{array}      \right.
$$ 
The elements $x_m$, $m\in \mathbb{Z}$ are called the cluster variables of $\mathcal{A}(b,c)$.  Fomin and Zelevinsky \cite{FZ} introduced cluster algebras and proved the Laurent phenomenon whose special case says that for every $m\in  \mathbb{Z}$ the cluster variable $x_m$ can be expressed as a Laurent polynomial of $x_1^{\pm 1}$ and  $x_2^{\pm 1}$. In addition, they conjectured that the coefficients of monomials in the Laurent expression of $x_m$ are non-negative integers.  When $bc\leq 4$, Sherman-Zelevinsky \cite{SZ} and independently Musiker-Propp \cite{MP} proved the conjecture. Moreover in this case the explicit combinatorial formulas for the coefficients are known.  In \cite{L}, we find a new  formula for the coefficients when $b=c\geq 2$.   

Before we state the main results of \cite{L}, we need some definitions.

\begin{defn}\label{modifiedbinomialcoeff}
For arbitrary (possibly negative) integers $A, B$, we define the modified binomial coefficient as follows.
$$\left[\begin{array}{c}{A } \\{B} \end{array}\right] := \left\{ \begin{array}{ll}  \prod_{i=0}^{A-B-1} \frac{A-i}{A-B-i}, & \text{ if }A > B\\ \, & \,  \\   1, & \text{ if }A=B \\ \, & \, \\  0, & \text{ if }A<B.  \end{array}  \right.$$ \qed
\end{defn} 

If $A\geq 0$ then $\left[\begin{array}{c}{A } \\{B} \end{array}\right]=\gchoose{A}{A-B}$ is just the usual binomial coefficient. In general, $\gchoose{A}{A-B}$ is equal to the generalized binomial coefficient ${A \choose B}$. But in this paper we use our modified binomial coefficients to  avoid too complicated expressions.

\begin{defn}
Let $\{a_n\}$ be the sequence  defined by the recurrence relation $$a_n=ca_{n-1} -a_{n-2},$$ with the initial condition $a_1=0$, $a_2=1$. If $c=2$ then $a_n=n-1$. When $c>2$, it is easy to see that 
$$
a_n= \frac{1}{\sqrt{c^2-4}  }\left(\frac{c+\sqrt{c^2-4}}{2}\right)^{n-1} - \frac{1}{\sqrt{c^2-4}  }\left(\frac{c-\sqrt{c^2-4}}{2}\right)^{n-1} = \sum_{i\geq 0} (-1)^i { {n-2-i} \choose i }c^{n-2-2i}.
$$ \qed
\end{defn}

Our main result in \cite{L} is the following.

\begin{thm}\label{mainthm} Assume that $b=c\geq 2$. Let $n\geq 3$. Then
\begin{equation}\label{mainformula}\aligned
&x_n= x_1^{-a_{n-1}} x_2^{-a_{n-2}} \sum_{e_1,e_2} \sum_{ t_0,t_1,\cdots,t_{n-4}} \left[ \left( \prod_{i=0}^{n-4} \left[\begin{array}{c}{{a_{i+1} - cs_i    } } \\{t_i} \end{array}\right]   \right) \right.\\
&\,\,\,\,\,\,\,\,\,\,\,\,\,\times \left.\gchoose{a_{n-2} - cs_{n-3} }{a_{n-2} - cs_{n-3}-e_2+s_{n-4}}\gchoose{-a_{n-3} + c e_2 }{  -a_{n-3} + c e_2-e_1+s_{n-3} } x_1^ {c(a_{n-2}-e_{2})} x_2^{ce_{1}}\right],
\endaligned\end{equation}
where $$
s_i=\sum_{j=0}^{i-1} a_{i-j+1}t_j,
$$
and the summations run over all integers $e_1,e_2, t_0,...,t_{n-4}$ satisfying 
\begin{equation}\label{cond501}\left\{
\begin{array}{l} 0\leq t_i \leq a_{i+1} - cs_i \, (0\leq i\leq n-4),\\
 0\leq a_{n-2} - cs_{n-3}-e_2+s_{n-4}\leq a_{n-2} - cs_{n-3}, \text{ and }  \\
e_2 a_{n-1} -e_1 a_{n-2}\geq 0.      
 \end{array} \right.\end{equation}
\end{thm}

Since $\gchoose{A}{B}\neq 0$ if and only if $A\geq B$, we may add the condition $0\geq -e_1+s_{n-3}$ to (\ref{cond501}). Then the summation in the statement is guaranteed to be a finite sum. A referee remarks that $F$-polynomials have similar expressions.  As he pointed out, the expression without (\ref{cond501}) is an easy consequence of the  formula (6.28) in the paper \cite{FZ4} by Fomin and Zelevinsky, and the one with $e_2 a_{n-1} -e_1 a_{n-2}\geq 0$  is a consequence of \cite[Proposition 3.5]{SZ} in the paper by Sherman and Zelevinsky. Our contribution is to show that  all the modified binomial coefficients in (\ref{mainformula}) except for the last one are non-negative.

As a corollary to Theorem~\ref{mainthm}, we obtain a new expression for the Euler-Poincar\'{e} characteristic of  the variety $\text{Gr}_{(e_1,e_2)}(M(n))$ of all subrepresentations of dimension $(e_1,e_2)$ in a unique (up to an isomorphism) indecomposable $Q_c$-representation $M(n)$ of dimension $(a_{n-1}, a_{n-2})$, where  $Q_c$ is the generalized Kronecker  quiver with  two vertices 1 and 2, and $c$ arrows from 1 to 2. We use a result of Caldero and  Zelevinsky \cite[Theorem 3.2 and (3.5)]{CZ}.

\begin{thm}[Caldero and  Zelevinsky]
The cluster variable $x_n$ is equal to 
$$
x_1^{-a_{n-1}} x_2^{-a_{n-2}}\sum_{e_1,e_2} \chi(\emph{Gr}_{(e_1,e_2)}(M(n)))x_1^ {c(a_{n-2}-e_{2})} x_2^{ce_{1}}.
$$
\end{thm}

\begin{cor}\label{maincor}
Assume that $b=c\geq 2$. For any $(e_1,e_2)$ and $n\geq 3$, the Euler-Poincar\'{e} characteristic of $\emph{Gr}_{(e_1,e_2)}(M(n))$ is equal to
\begin{equation}\label{submainformula}\sum_{ t_0,t_1,\cdots,t_{n-4}} \left[ \left( \prod_{i=0}^{n-4} \left[\begin{array}{c}{{a_{i+1} - cs_i    } } \\{t_i} \end{array}\right]   \right)\gchoose{a_{n-2} - cs_{n-3} }{a_{n-2} - cs_{n-3}-e_2+s_{n-4}}\gchoose{-a_{n-3} + c e_2 }{  -a_{n-3} + c e_2-e_1+s_{n-3} } \right],   
\end{equation}
where  the summation runs over all integers $t_0,...,t_{n-4}$ satisfying 
\begin{equation}\label{cond512}\left\{
\begin{array}{l} 0\leq t_i \leq a_{i+1} - cs_i \, (0\leq i\leq n-4), \text{ and } \\
  0\leq a_{n-2} - cs_{n-3}-e_2+s_{n-4}\leq a_{n-2} - cs_{n-3}.      
 \end{array} \right.\end{equation}\end{cor}
 \begin{proof}
 Corollary~\ref{maincor} is an immediate consequence of Theorem~\ref{mainthm} thanks to the result of Caldero and  Zelevinsky \cite[Theorem 3.2 and (3.5)]{CZ}.
 \end{proof}

\begin{cor}\label{maincor2}
Assume that $b=c\geq 3$. Let $n\geq 3$. For any $(e_1,e_2)$ with $e_2\geq \frac{a_{n-3}}{c}$, the Euler-Poincar\'{e} characteristic of $\emph{Gr}_{(e_1,e_2)}(M(n))$ is non-negative.
\end{cor}
\begin{proof}
By (\ref{cond512}), all the modified binomial coefficients except for the last one in (\ref{submainformula}) are non-negative. If $e_2\geq \frac{a_{n-3}}{c}$ then the last one also becomes non-negative. Therefore, Corollary~\ref{maincor} implies that $\chi(\text{Gr}_{(e_1,e_2)} M(n))$ is non-negative.
\end{proof}

In order to prove (or disprove) that the Euler characteristic of $\text{Gr}_{(e_1,e_2)}(M(n))$ is non-negative for $0<e_2< \frac{a_{n-3}}{c}$, we need to find another expression for the Euler characteristic, preferably an expression which could explain  a cell decomposition of $\text{Gr}_{(e_1,e_2)}(M(n))$. Conjecturally we have a better expression for this purpose, especially when $e_1$ is small.

\begin{lem}\label{e1small}
Assume that $b=c\geq 2$. If $e_1<c$ and $n\geq 4$, then the Euler-Poincar\'{e} characteristic of $\emph{Gr}_{(e_1,e_2)}(M(n))$ is equal to  
\begin{equation}\label{e1smallformula}\sum_{t_{n-4} } \gchoose{a_{n-3} }{t_{n-4}}\gchoose{a_{n-2} - ct_{n-4} }{a_{n-2} - ct_{n-4}-e_2}\gchoose{-a_{n-3} + c e_2 }{  -a_{n-3} + c e_2-e_1+t_{n-4} }.\end{equation}
\end{lem}
\begin{proof}
By Corollary~\ref{maincor}, the Euler characteristic is equal to (\ref{submainformula}), where $ t_i \geq 0 \, (0\leq i\leq n-4)$.  If $t_i\geq 1$ for some $0\leq i\leq n-5$, then $s_{n-3}=\sum_{j=0}^{n-4} a_{n-2-j}t_j\geq a_3=c>e_1$, which implies $\gchoose{-a_{n-3} + c e_2 }{  -a_{n-3} + c e_2-e_1+s_{n-3} }=0.$ So we can assume that $t_i=0$ for $0\leq i\leq n-5$. Then we have  $s_{n-4}=0$, $s_{n-3}=t_{n-4}$, and all the modified binomial coefficients except for the last three are 1. Therefore, (\ref{submainformula}) reduces to (\ref{e1smallformula}).
\end{proof}

We will need the following standard fact later, whose proof will be omitted.

\begin{lem}\label{stfact}
Let $A,B$ be any $($possibly negative$)$ integers, and let $n$ be any positive integer. Then 
$$\frac{d^{n}}{dy^{n}}(1+y^{A})^{B}=n!\sum_{i=1}^n \sum_{\tiny{\begin{array}{c} j_1+\cdots+j_i=n\\j_1,\cdots,j_i\geq 1  \end{array}}}\gchoose{B}{B-i}(1+y^{A})^{B-i} \gchoose{A}{A-j_1}\cdots\gchoose{A}{A-j_i}y^{Ai-n}.$$
\end{lem}

\begin{lem}
Assume that $b=c\geq 2$. If $e_1<c$ and $n\geq 4$, then the Euler-Poincar\'{e} characteristic of $\emph{Gr}_{(e_1,e_2)}(M(n))$ is equal to
\begin{equation}\label{0809eq02}
{{ce_2}\choose {e_1}}{{a_{n-2}}\choose{e_2}} + \sum_{k=1}^{e_2}\sum_{i=1}^k \sum_{\tiny{\begin{array}{c} j_1+\cdots+j_i=k\\j_1,\cdots,j_i\geq 1  \end{array} }} \tiny{{{a_{n-3}} \choose i}{{ce_2-i}\choose {e_1-i}} \gchoose{-c}{-c-j_1}\cdots\gchoose{-c}{-c-j_i}{{a_{n-2}}\choose{e_2-k}}}.
 \end{equation}
\end{lem}
\begin{proof}
We want to show that (\ref{e1smallformula}) is equal to (\ref{0809eq02}). We start with the following binomial formula:
$$(1+y^{-c})^{a_{n-3}} y^{a_{n-2}} = \sum_i {{a_{n-3}} \choose i}y^{a_{n-2}-ci}.$$
By taking the $e_2$-th derivative, we get
\begin{equation}\label{0805eq1}\frac{1}{e_2!}\frac{d^{e_2}}{dy^{e_2}}\left[(1+y^{-c})^{a_{n-3}} y^{a_{n-2}}\right] = \sum_i {{a_{n-3}} \choose i}
\gchoose{a_{n-2}-ci}{a_{n-2}-ci-e_2}y^{a_{n-2}-ci-e_2}.\end{equation}
Then we multiply (\ref{0805eq1}) by
$$
(1+y^{-c})^{-a_{n-3}+e_2c}=\sum_j \gchoose{-a_{n-3}+e_2c}{-a_{n-3}+e_2c-e_1+j}(y^{-c})^{e_1-j},
$$which yields
\begin{equation}\label{0809eq03}
\aligned &(1+y^{-c})^{-a_{n-3}+e_2c}\frac{1}{e_2!}\frac{d^{e_2}}{dy^{e_2}}\left[(1+y^{-c})^{a_{n-3}} y^{a_{n-2}}\right]\\
&=\sum_i {{a_{n-3}} \choose i}
\gchoose{a_{n-2}-ci}{a_{n-2}-ci-e_2}y^{a_{n-2}-ci-e_2}\sum_j \gchoose{-a_{n-3}+e_2c}{-a_{n-3}+e_2c-e_1+j}(y^{-c})^{e_1-j}\\
&=\sum_{i,j} {{a_{n-3}} \choose i}
\gchoose{a_{n-2}-ci}{a_{n-2}-ci-e_2}\gchoose{-a_{n-3}+e_2c}{-a_{n-3}+e_2c-e_1+j}y^{a_{n-2}-c(e_1+i-j)-e_2}.
\endaligned\end{equation}

On the other hand, Lemma~\ref{stfact} implies that
\begin{equation}\label{0809eq05}
\aligned &\frac{1}{e_2!}\frac{d^{e_2}}{dy^{e_2}}\left[(1+y^{-c})^{a_{n-3}} y^{a_{n-2}}\right]\\
&= (1+y^{-c})^{a_{n-3}} {{a_{n-2}}\choose{e_2}}y^{a_{n-2}-e_2}\\
&\,\,\,\,\,\, +\sum_{k=1}^{e_2}\sum_{i=1}^k \sum_{\tiny{\begin{array}{c} j_1+\cdots+j_i=k\\j_1,\cdots,j_i\geq 1  \end{array} }} \tiny{{{a_{n-3}} \choose i}(1+y^{-c})^{a_{n-3}-i} \gchoose{-c}{-c-j_1}\cdots\gchoose{-c}{-c-j_i}y^{-ci-k}{{a_{n-2}}\choose{e_2-k}}y^{a_{n-2}-e_2+k}}.
\endaligned
\end{equation}
Combining (\ref{0809eq03}) and  (\ref{0809eq05}),  we get
$$\aligned
&\sum_{i,j} {{a_{n-3}} \choose i}
\gchoose{a_{n-2}-ci}{a_{n-2}-ci-e_2}\gchoose{-a_{n-3}+e_2c}{-a_{n-3}+e_2c-e_1+j}y^{a_{n-2}-c(e_1+i-j)-e_2}\\
&= (1+y^{-c})^{ce_2} {{a_{n-2}}\choose{e_2}}y^{a_{n-2}-e_2}\\
&\,\,\,\,\,\, +\sum_{k=1}^{e_2}\sum_{i=1}^k \sum_{\tiny{\begin{array}{c} j_1+\cdots+j_i=k\\j_1,\cdots,j_i\geq 1  \end{array} }} \tiny{{{a_{n-3}} \choose i}(1+y^{-c})^{ce_2-i} \gchoose{-c}{-c-j_1}\cdots\gchoose{-c}{-c-j_i}{{a_{n-2}}\choose{e_2-k}}y^{a_{n-2}-e_2-ci}}.
\endaligned$$

Comparing the coefficients of $y^{a_{n-2}-ce_1-e_2}$ in both sides, we obtain 
\begin{equation}\label{0805eq07}
\aligned &\sum_{i} {{a_{n-3}} \choose i}
\gchoose{a_{n-2}-ci}{a_{n-2}-ci-e_2}\gchoose{-a_{n-3}+e_2c}{-a_{n-3}+e_2c-e_1+i} \\
&={{ce_2}\choose {e_1}}{{a_{n-2}}\choose{e_2}} + \sum_{k=1}^{e_2}\sum_{i=1}^k \sum_{\tiny{\begin{array}{c} j_1+\cdots+j_i=k\\j_1,\cdots,j_i\geq 1  \end{array} }} \tiny{{{a_{n-3}} \choose i}{{ce_2-i}\choose {e_1-i}} \gchoose{-c}{-c-j_1}\cdots\gchoose{-c}{-c-j_i}{{a_{n-2}}\choose{e_2-k}}}.
\endaligned \end{equation}Then the desired statement follows from Lemma~\ref{e1small}.
\end{proof}

Now the EC-polynomial is expected to come into play. 
\begin{conj}\label{ECecconj}
Assume that $b=c\geq 2$. If $e_1<c$ and $n\geq 4$, then the Euler-Poincar\'{e} characteristic of $\emph{Gr}_{(e_1,e_2)}(M(n))$ is equal to
\begin{equation}\label{0809eq09}\aligned
 \frac{1}{(e_2!)^2}\sum_{z_1+\cdots+z_{e_2}=e_1}{c\choose {z_1}}\cdots{c\choose {z_{e_2}}}EC_{e_2}(\{z_1,\cdots,z_{e_2}\};-a_{n-2},-a_{n-3},c).
\endaligned \end{equation}
\end{conj}
To prove Conjecture~\ref{ECecconj}, one needs to show that $(\ref{0809eq02})=(\ref{0809eq09})$. In fact, the case $e_2\leq 2$ is elementary, from which the author guessed the general case and checked $(\ref{0809eq02})=(\ref{0809eq09})$ when $e_2\leq 5$.

\noindent\emph{Acknowledgement.} We are very much grateful to anonymous referees for their useful suggestions and helpful comments on \cite{L}.

\end{document}